\documentclass[12pt,a4paper]{amsart}
\usepackage{amscd,amsmath,amsthm,amssymb}
\usepackage[colorlinks=true, urlcolor=blue, linkcolor=blue, citecolor=blue]{hyperref}
\def\NZQ{\mathbb}               

\def\KK{{\NZQ K}}

\def\Rc{{\mathcal R}}

\usepackage{tikz}
\usetikzlibrary{arrows.meta, positioning}
\def\opn#1#2{\def#1{\operatorname{#2}}} 
	\opn\chara{char} \opn\length{\ell} \opn\pd{pd} \opn\rk{rk}
	\opn\projdim{proj\,dim} \opn\injdim{inj\,dim} \opn\rank{rank}
	\opn\depth{depth} \opn\grade{grade} \opn\height{height}
	\opn\embdim{emb\,dim} \opn\codim{codim}
	\opn\Cl{Cl}
	
	\opn\Tr{Tr} \opn\bigrank{big\,rank}
	\opn\superheight{superheight}\opn\lcm{lcm}
	\opn\trdeg{tr\,deg}
	\opn\rdeg{rdeg}
	\opn\reg{reg} \opn\lreg{lreg} \opn\ini{in} \opn\lpd{lpd}
	\opn\size{size} \opn\sdepth{sdepth}
	\opn\link{link}\opn\fdepth{fdepth}\opn\lex{lex}
	\opn\tr{tr}
	\opn\type{type}
	\opn\gap{gap}
	\opn\arithdeg{arith-deg}
	\opn\revlex{revlex}
	\opn\rev{rev}
	\opn\lex{lex}
	\DeclareMathOperator{\Gr}{Gr}
	\opn\div{div} \opn\Div{Div} \opn\cl{cl} \opn\Cl{Cl}
	\opn\Spec{Spec} \opn\Supp{Supp} \opn\supp{supp} \opn\Sing{Sing}
	\opn\Ass{Ass} \opn\Min{Min}\opn\Mon{Mon}
	\opn\Ann{Ann} \opn\Rad{Rad} \opn\Soc{Soc}
	\opn\Im{Im} \opn\Ker{Ker} \opn\Coker{Coker} \opn\Am{Am}
	\opn\Hom{Hom} \opn\Tor{Tor} \opn\Ext{Ext} \opn\End{End}
	\opn\Aut{Aut} \opn\id{id}
	
	\opn\nat{nat}
	\opn\pff{pf}
	\opn\Pf{Pf} \opn\GL{GL} \opn\SL{SL} \opn\mod{mod} \opn\ord{ord}
	\opn\Gin{Gin} \opn\Hilb{Hilb}\opn\sort{sort}
	\opn\PF{PF}\opn\Ap{Ap}
	\opn\mult{mult}
	\opn\bight{bight}
	\opn\div{div}
	\opn\Div{Div}
	\opn\aff{aff}
	\opn\relint{relint} \opn\st{st}
	\opn\lk{lk} \opn\cn{cn} \opn\core{core} \opn\vol{vol}  \opn\inp{inp} 
	\opn\nilpot{nilpot}
	\opn\link{link} \opn\star{star}\opn\lex{lex}\opn\set{set}
	\opn\width{wd}
	\opn\Fr{F}
	\opn\QF{QF}
	\opn\G{G}
	\opn\type{type}\opn\res{res}
	\opn\conv{conv}
	\opn\Int{Int}
	\opn\Deg{Deg}
	\opn\Sym{Sym}
	\opn\Con{Con}
	\opn\gr{gr}
	
	\def\pot#1#2{#1[\kern-0.28ex[#2]\kern-0.28ex]}

	\opn\dirlim{\underrightarrow{\lim}}
	\opn\inivlim{\underleftarrow{\lim}}
	\let\to=\rightarrow
	
	\def\Implies{\ifmmode\Longrightarrow \else
		\unskip${}\Longrightarrow{}$\ignorespaces\fi}
	\def\implies{\ifmmode\Rightarrow \else
		\unskip${}\Rightarrow{}$\ignorespaces\fi}
	\def\iff{\ifmmode\Longleftrightarrow \else
		\unskip${}\Longleftrightarrow{}$\ignorespaces\fi}

	\let\:=\colon
	\newtheorem{Theorem}{Theorem}[section]
	\newtheorem{Lemma}[Theorem]{Lemma}
	\newtheorem{Corollary}[Theorem]{Corollary}
	\newtheorem{Proposition}[Theorem]{Proposition}
	\theoremstyle{definition}
	\newtheorem{Remark}[Theorem]{Remark}
	\newtheorem{Definition}[Theorem]{Definition}
	\newtheorem{Example}[Theorem]{Example}

	\let\epsilon\varepsilon
	\let\kappa=\varkappa
	\opn\dis{dis}
	\def\pnt{{\raise0.5mm\hbox{\large\bf.}}}
	
	\opn\Lex{Lex}
	
\begin{document}
\title[algebras with straightening laws]{ Elimination ideals of Pl\"ucker ideals and algebras with straightening laws}
\author[V.~Borovik]{Viktoriia Borovik}
\author[T.~Hibi]{Takayuki Hibi}
\address{(Viktoriia Borovik)
Max Planck Institute for Mathematics in the Sciences, Inselstrasse 22, 04103 Leipzig, Germany
}
\email{viktoriia.borovik@mis.mpg.de}
\address{(Takayuki Hibi) Department of Pure and Applied Mathematics, Graduate School of Information Science and Technology, Osaka University, Suita, Osaka 565--0871, Japan}
\email{hibi@math.sci.osaka-u.ac.jp}




\begin{abstract}
It is well known that the Plücker ideal defining the Grassmannian is  generated by quadratic Plücker relations. These relations form a reverse lexicographic Gröbner basis and endow the Plücker algebra with the structure of an algebra with straightening laws (ASL). In this paper, we study quadratically gene\-rated projections of the Grassmannian of lines $\mathrm{Gr}(2,n)$. 
We then combinatorially characterize the Gorenstein ASL subalgebras of the  Plücker algebra of $\mathrm{Gr}(2,n)$.
\end{abstract}

\maketitle
\thispagestyle{empty}

\section*{Introduction}
Binomials and monomials have long been central objects of study in combinatorics and commutative algebra.  In the early days, squarefree monomial ideals, the so-called Stanley-Reisner ideals \cite{Reisner76, Stanley75}, associated with simplicial complexes and with finite partially ordered sets, were thoroughly investigated. In the late 1990's, subalgebras of a polynomial ring over a field generated by monomials, and their defining ideals, the so-called toric ideals \cite{OH99, Sturmfels96}, became fashionable with the help of Gr\"obner bases.  
In the early 2010's, the study of powers of edge ideals and of binomial edge ideals \cite{HHZ04, H+10} quickly developed into a large subarea.

In this paper, following classical works of Pl\"ucker \cite{Pluecker1, Pluecker2} and Grassmann \cite{Grassmann} and from the viewpoint of algebras with straightening laws \cite{Eisenbud80}, we study a subalgebra of the polynomial ring generated by binomials whose defining ideal is an elimination ideal (see, for example, \cite[Theorem 1.4.1]{GB2013}) of the Pl\"ucker ideal.

\smallskip

Let $A=\mathbb{K}[x_1,\ldots,x_n,y_1,\ldots,y_n]$ denote the polynomial ring in $2n$ variables over a field $\mathbb{K}$, and let
$L_n=\{\,p_{ij}: 1\le i<j\le n\,\}$ be the distributive lattice with partial order defined by $p_{ij}\le p_{k\ell}$ if $i\le k$ and $j\le \ell$. The homogeneous coordinate ring
$\Rc_{\mathbb{K}}[L_n]=\mathbb{K}[\,x_i y_j - x_j y_i: 1\le i<j\le n\,]$
of the Grassmannian $\Gr(2,n)$ is an algebra with straightening laws on $L_n$ over $\mathbb{K}$, via the identification $p_{ij}\mapsto x_i y_j - x_j y_i$. The Pl\"ucker ideal $I_{L_n}$, which is the defining ideal of $\Rc_{\mathbb{K}}[L_n]$ in the polynomial ring $\mathbb{K}[\,p_{ij}: 1\le i<j\le n\,]$, is generated by the quadratic Pl\"ucker relations
\begin{equation}\label{eq: Plucker_relations intro}
    Q_{ijk\ell}:=p_{i\ell}p_{jk}-p_{ik}p_{j\ell}+p_{ij}p_{k\ell}, 
\qquad 1\le i<j<k<\ell\le n.
\end{equation}
Our goal is to find a sublattice $L\subset L_n$ with the following properties:
\begin{itemize}
    \item[(i)] the subalgebra $\Rc_{\mathbb{K}}[L]\subset \Rc_{\mathbb{K}}[L_n]$ generated by those $p_{ij}=x_i y_j - x_j y_i$ with $p_{ij}\in L$ is an algebra with straightening laws on $L$ over $\mathbb{K}$;
\item[(ii)] the defining ideal $I_L$ of $\Rc_{\mathbb{K}}[L]$ in the polynomial ring $\mathbb{K}[\,p_{ij}: p_{ij}\in L\,]$ is an elimination ideal of $I_{L_n}$;
\item[(iii)] $\Rc_{\mathbb{K}}[L]$ is Gorenstein.
\end{itemize}

The outline of the present paper is as follows. In Section \ref{sec: 1}, we briefly review algebras with straightening laws. In Section \ref{sec: 2}, we introduce a specific lexicographic order $<_{\lex}$ for which the quadratic Plücker relations \eqref{eq: Plucker_relations intro} form a Gröbner basis of the Plücker ideal $I_{L_n}$. We then study distinguished sublattices $L$ of $L_n$ (\emph{compatible} and \emph{perfect compatible} sublattices)  such that the defining ideal $I_L$ of the subalgebra $\Rc_\mathbb{K}[L]$ of $\Rc_\mathbb{K}[L_n]$ is the elimination ideal $I_{L_n} \cap \mathbb{K}[p_{ij} : p_{ij} \in L]$ of the Pl\"ucker ideal.  

Although our primary objects are sublattices of $L_n$, the combinatorial aspects can be described elegantly via finite graphs. In Section~\ref{sec: 3}, to each sublattice $L \subseteq L_n$ we associate a finite graph $G_L$ on the vertex set $[n]=\{1,\ldots,n\}$, with edges $\{i,j\}$ precisely when $p_{ij} \in L$. Our discussion focuses on interval graphs. In Theorem~\ref{Th_Gorenstein}, we provide an effective criterion for $\Rc_{\mathbb{K}}[L]$ to be Gorenstein when $L$ is a compatible sublattice. Furthermore, using this criterion, we compute the number of perfect compatible sublattices $L \subseteq L_n$ for which $\Rc_{\mathbb{K}}[L]$ is Gorenstein (Corollary~\ref{enumeration_2}).

In the Appendix \ref{appendix}, we construct another lexicographic Gröbner basis of the Plücker ideal \(I_{L_n}\) consisting of quadrics \eqref{eq: Plucker_relations intro} and cubics. This yields an additional combinatorial interpretation of the Catalan numbers in terms of nested arc arrangements on \(n\) points on a line. Using this Gröbner basis, we also partially repeat  the elimination results from Section~\ref{sec: 2}.

\section{Algebras with straightening laws}\label{sec: 1}
Let $R = \bigoplus_{n=0}^{\infty} R_n$ be a noetherian graded algebra over a field ${\mathbb K}$.  Let $P$ be a finite poset (partially ordered set) and suppose we are given an injection
\[
\varphi: P \hookrightarrow \bigcup_{n=1}^{\infty} R_n
\]
such that $R$ is generated as a ${\mathbb K}$-algebra by $\varphi(P)$.  A {\em standard monomial} is a homogeneous element of $R$ of the form $\varphi(\gamma_1) \varphi(\gamma_2)\cdots \varphi(\gamma_n)$, where $\gamma_1 \leq \gamma_2 \leq \cdots \leq \gamma_n$ in~$P$ (with the empty product interpreted as $1$).  We say that $R$ is an \emph{algebra with straightening laws} \cite{Eisenbud80} on $P$ over $\mathbb{K}$  if the following conditions hold:
\begin{itemize}
\item[]
(ASL\,-1)
The set of standard monomials forms a $\mathbb{K}$-basis of $R$.
\item[]
(ASL\,-2)
If $\alpha$ and $\beta$  are incomparable in $P$ and if
\begin{eqnarray}
\label{ASL}
\, \, \, \, \, \, \, \, \, \, \, \, \, \, \, \, \, \, \, \, 
\varphi(\alpha)\varphi(\beta) 
= \sum_{i} r_i\,\varphi(\gamma_{i_1}) \varphi(\gamma_{i_2})\cdots \varphi(\gamma_{i_{n_i}}), \, \text{ with } 0 \neq r_i \in {\mathbb K}, \, \, \, \gamma_{i_1}\leq \gamma_{i_2} \leq \cdots 
\end{eqnarray}
is the unique expression for $\varphi(\alpha)\varphi(\beta) \in R$ as a $\mathbb{K}$-linear combination of distinct standard monomials guaranteed by (ASL-1), then $\forall \, i: \gamma_{i_1} \leq \alpha, \beta$. 
\end{itemize}
The right-hand side of the relation in (ASL-2) is allowed to be the empty sum (i.e., $0$). We abbreviate ``algebra with straightening laws'' as ASL. The relations in (ASL-2) are called the \emph{straightening relations} for an algebra $R$.

Let $S=\mathbb{K}[x_{\alpha} : \alpha \in P]$ denote the polynomial ring in $|P|$ many variables over the field $\mathbb{K}$, and define the surjection $\pi: S \to R$ by $\pi(x_\alpha) = \varphi(\alpha)$. The defining ideal $I_R$ of $R=\bigoplus_{n=0}^{\infty} R_n$ is the kernel $\ker(\pi)$ of $\pi$. If $\alpha,\beta \in P$ are incomparable, define
\[
f_{\alpha,\beta}
:= x_\alpha x_\beta - \sum_{i} r_i\,x_{\gamma_{i_1}}x_{\gamma_{i_2}} \cdots x_{\gamma_{i_{n_i}}}, \, \text{ with } 0 \neq r_i \in {\mathbb K}, \, \, \, \gamma_{i_1}\leq \gamma_{i_2} \leq \cdots,
\]
arising from (ASL-2). Then we have that $f_{\alpha,\beta} \in I_R$. Let $G_R$ denote the set of all polynomials $f_{\alpha,\beta}$ with $\alpha,\beta$ incomparable in $P$. Let $<_{\mathrm{rev}}$ denote the reverse lexicographic order \cite[Example~2.1.2(b)]{HHgtm260} on $S$ induced by an ordering of the variables such that $x_{\alpha} <_{\mathrm{rev}} x_{\beta}$ whenever $\alpha < \beta$ in $P$. It follows from (ASL-1) that $G_R$ is a Gröbner basis of $I_R$ with respect to $<_{\mathrm{rev}}$. In particular, $G_R$ generates $I_R$.

\smallskip

We briefly recall some basics on finite posets. A \emph{chain} in a finite poset $P$ is a totally ordered subset of $P$. The \emph{length} of a chain $C$ is $|C|-1$, where $|C|$ denotes the cardinality of $C$. The \emph{rank} of $P$ is the largest length of a chain in $P$; we denote it by $\rank(P)$. A finite poset is called \emph{pure} if all of its maximal chains have the same length. If $R$ is an ASL on $P$ over $\mathbb{K}$, then $\dim R=\rank(P)+1$.

\section{Pl\"ucker ideals and their elimination ideals}\label{sec: 2}
Let $A=\mathbb{K}[x_1,\ldots,x_n,y_1,\ldots,y_n]$ be the polynomial ring in $2n$ variables over a~field $\mathbb{K}$, and set
\[
p_{ij}:=x_i y_j - x_j y_i,\qquad \text{for } \; 1\le i<j\le n.
\]
Let $L_n =\{\,p_{ij} : 1\le i<j\le n\,\}$, and define a partial order on $L_n$ by $p_{ij}\le p_{k\ell}$ if $i\le k$ and $j\le \ell$. It is easy to see that $L_n$ is a distributive lattice \cite{Hibi87}. 
\begin{Example}
The Hasse diagram of the distributive lattice \(L_5\) is shown in Figure~\ref{fig: dist lattice L5}. We abuse notation by sometimes identifying \(L_n\) with the set of indices \(ij\) of the Plücker variables $p_{ij}$.
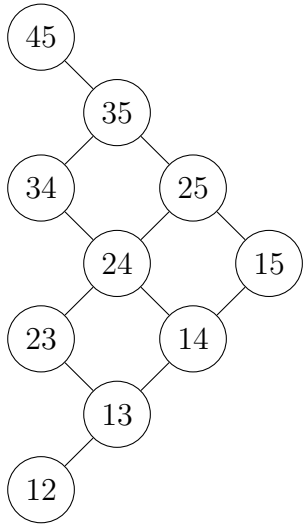
\begin{figure}
    \centering
\begin{tikzpicture}
\begin{scope}[rotate around z=0]
\node[draw,shape=circle] (12) at (0,0) {12};
\node[draw,shape=circle] (13) at (1,1) {13};
\node[draw,shape=circle] (14) at (2,2) {14};
\node[draw,shape=circle] (15) at (3,3) {15};
\node[draw,shape=circle] (23) at (0,2) {23};
\node[draw,shape=circle] (24) at (1,3) {24};
\node[draw,shape=circle] (25) at (2,4) {25};
\node[draw,shape=circle] (34) at (0,4) {34};
\node[draw,shape=circle] (35) at (1,5) {35};
\node[draw,shape=circle] (45) at (0,6) {45};
\draw (12)--(13)--(14)--(15);
\draw (23)--(24)--(25);
\draw (34)--(35);
\draw (13)--(23);
\draw (14)--(24)--(34);
\draw (15)--(25)--(35)--(45);
\end{scope}
\end{tikzpicture}
    \caption{The Hasse diagram of the distributive lattice $L_5$. }
    \label{fig: dist lattice L5}
\end{figure}    
\end{Example}
Let $\Rc_{\mathbb{K}}[L_n]$ denote the subalgebra of $A$ generated by the $\binom{n}{2}$ binomials $p_{ij}$ with $1\le i<j\le n$. In other words,
\[
\Rc_{\mathbb{K}}[L_n]=\mathbb{K}\big[\,x_i y_j - x_j y_i : 1\le i<j\le n\,\big]
\]
is the homogeneous coordinate ring of the Grassmannian of lines $\Gr(2,n)$.
It follows from \cite{DEP80} that $\Rc_\mathbb{K}[L_n]$ is an
ASL on $L_n$ over $\mathbb{K}$ whose (ASL-2) relations are 
\begin{eqnarray}
\label{(ASL-2)}
    p_{il}p_{jk} = p_{ik}p_{jl} - p_{ij}p_{kl}, \qquad 1\leq i<j<k<\ell\leq n.
\end{eqnarray}
The {\em Pl\"ucker ideal} $I_{L_n}$, which is the defining ideal of $\Rc_\mathbb{K}[L_n]$ in the polynomial ring $S=\mathbb{K}[\,p_{ij} : 1 \leq i < j \leq n]$ in ${n \choose 2}$ variables over $\mathbb{K}$, is generated by the quadratic {\em Pl\"ucker relations}
\begin{equation}\label{quadratic}
    Q_{ijk\ell} = p_{i\ell}p_{jk} - p_{ik}p_{j\ell} + p_{ij}p_{k\ell}, 
\qquad 1\leq i<j<k<\ell\leq n.  
\end{equation}
Let $<_{\rev}$ denote the reverse lexicographic order on $S$ induced by the variable order given by any linear extension of $L_n$. 
Since $\Rc_\mathbb{K}[L_n]$ is an ASL on $L_n$ over $\mathbb{K}$, the  quadratic Pl\"ucker relations \eqref{quadratic}
form a Gr\"obner basis of $I_{L_n}$ with respect to $<_{\rev}$.  

\smallskip 

On the other hand, we introduce the finite poset $\Pi_n=\{\,p_{ij} : 1 \leq i < j \leq n \}$ whose partial order is defined by setting $p_{ij} \leq p_{k\ell}$ if $i \leq k$ and $j \geq \ell$.  An example of this poset for $n=5$ is shown in Figure~\ref{fig: HassePi5}. Similarly, let $<_{\lex}$ denote the lexicographic order \cite[Example 2.1.2(a)]{HHgtm260} on $S$ induced by the variable order given by any linear extension of $\Pi_n$, that is, $p_{k \ell} <_{\lex} p_{ij}$ if $p_{ij} <_{\Pi_n} p_{k\ell}$.

\begin{figure}
    \centering
\begin{tikzpicture}
\begin{scope}[rotate around z=-90]
\node[draw,shape=circle] (12) at (0,0) {12};
\node[draw,shape=circle] (13) at (1,1) {13};
\node[draw,shape=circle] (14) at (2,2) {14};
\node[draw,shape=circle] (15) at (3,3) {15};
\node[draw,shape=circle] (23) at (0,2) {23};
\node[draw,shape=circle] (24) at (1,3) {24};
\node[draw,shape=circle] (25) at (2,4) {25};
\node[draw,shape=circle] (34) at (0,4) {34};
\node[draw,shape=circle] (35) at (1,5) {35};
\node[draw,shape=circle] (45) at (0,6) {45};
\draw (12)--(13)--(14)--(15);
\draw (23)--(24)--(25);
\draw (34)--(35);
\draw (13)--(23);
\draw (14)--(24)--(34);
\draw (15)--(25)--(35)--(45);
\end{scope}
\end{tikzpicture}
    \caption{The Hasse diagram of the poset $\Pi_5$.}
    \label{fig: HassePi5}
\end{figure}
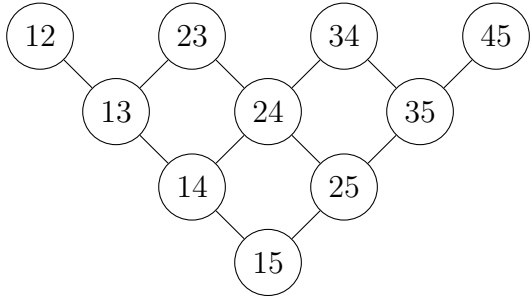

\begin{Lemma}
\label{lex}
The set of quadratic Pl\"ucker relations (\ref{quadratic}) is a Gr\"obner basis of the Plücker ideal $I_{L_n}$ with respect to $<_{\lex}$. 
\end{Lemma}

\begin{proof}
Since the initial ideal ${\rm in}_{<_{\rev}}(I_{L_n})$ of $I_{L_n}$ with respect to $<_{\rev}$ coincides with the monomial ideal $({\rm in}_{<_{\lex}}(Q_{ijk\ell}): 1\leq i<j<k<l\leq n)$ and since the set of quadratic Pl\"ucker relations (\ref{quadratic}) is a Gr\"obner basis of $I_{L_n}$ with respect to $<_{\rev}$, it follows that the set of quadratic Pl\"ucker relations (\ref{quadratic}) is a Gr\"obner basis of $I_{L_n}$ with respect to $<_{\lex}$.   
\end{proof}

We now pass to the elimination problem. Let $L$ be a sublattice of $L_n$ and $\Rc_\mathbb{K}[L]$ the subalgebra of $\Rc_\mathbb{K}[L_n]$ which is generated by those binomials $p_{ij}=x_iy_j-x_jy_i$ with $p_{ij} \in L$.  Let $I_L$ denote the defining ideal of $\Rc_\mathbb{K}[L]$ in the polynomial ring $\mathbb{K}[p_{ij} : p_{ij} \in L]$ in $|L|$ many variables over $\mathbb{K}$. 

\begin{Lemma}
    \label{sublattice}
If $L$ is a sublattice of $L_n$, then $\Rc_\mathbb{K}[L]$ is an ASL on $L$ over $\mathbb{K}$ if and only if the following condition is satisfied:  if $p_{i\ell}$ and $p_{jk}$ with $i < j < k < \ell$ belong to $L$, then also $p_{ij}$ and $p_{k\ell}$ belong to $L$.  \end{Lemma}

\begin{proof}
 \noindent
{\bf (``if'')} 
Let $p_{i\ell}$ and $p_{jk}$ with $i < j < k < \ell$ belong to $L$.  Since $L$ is a sublattice, both $p_{ik} = p_{i\ell} \wedge p_{jk}$ and $p_{j \ell} = p_{i\ell} \vee p_{jk}$ belong to $L$.  Since $p_{ij}$ and $p_{k\ell}$ belong to $L$, the (ASL-2) relation (\ref{(ASL-2)}) remains in $L$.  Thus $\Rc_\mathbb{K}[L]$ is an ASL on $L$ over $\mathbb{K}$.
\smallskip

\noindent
{\bf (``only if'')} Let $p_{i\ell}$ and $p_{jk}$ with $i < j < k < \ell$ belong to $L$ and suppose that $p_{ij}$ or $p_{k\ell}$ does not belong to $L$.  Let us express $p_{i\ell}p_{jk}$ as a linear combination of standard monomials of $\Rc_\mathbb{K}[L]$, say, 
\[
p_{i\ell}p_{jk} = \sum_{q=1}^{t} c_{q}p_{i_qj_q}p_{k_q\ell_q}, \quad 0 \neq c_{q} \in \KK.
\]
Then, by using (\ref{(ASL-2)}), in $\Rc_\mathbb{K}[L_n]$ one has
\begin{eqnarray}
\label{123456789}   
p_{ik}p_{j\ell} - p_{ij}p_{k\ell} = \sum_{q=1}^{t} c_{q}p_{i_qj_q}p_{k_q\ell_q}.
\end{eqnarray}
However, (\ref{123456789}) contradicts (ASL-1) of $\Rc_\mathbb{K}[L_n]$, as $p_{ij}$ or $p_{k\ell}$ does not appear in the right-hand side of (\ref{123456789}).  Thus $\Rc_\mathbb{K}[L]$ cannot be an ASL on $L$ over $\mathbb{K}$.
\end{proof}

A \emph{poset ideal} of a finite poset \(P\) is a subset \(J \subseteq P\) with the property that if \(\alpha \in J\) and \(\beta \in P\) satisfy \(\beta < \alpha\), then \(\beta \in J\).

\begin{Definition}
We say that a sublattice \(L\) of \(L_n\) is \emph{compatible} with \(\Pi_n\) if \(L\) has rank \(n\) and \(L_n \setminus L\) is a poset ideal of \(\Pi_n\).
\end{Definition}

\begin{figure}
    \centering
\begin{tikzpicture}
\begin{scope}[rotate around z=0]
\node[draw,shape=circle] (12) at (0,0) {12};
\node[draw,shape=circle] (13) at (1,1) {13};
\node[draw,shape=circle] (14) at (2,2) {14};
\node[draw,shape=circle] (15) at (3,3) {15};
\node[draw,shape=circle] (23) at (0,2) {23};
\node[draw,shape=circle] (24) at (1,3) {24};
\node[draw,shape=circle] (25) at (2,4) {25};
\node[draw,shape=circle] (26) at (3,5) {26};
\node[draw,shape=circle] (34) at (0,4) {34};
\node[draw,shape=circle] (35) at (1,5) {35};
\node[draw,shape=circle] (36) at (2,6) {36};
\node[draw,shape=circle] (45) at (0,6) {45};
\node[draw,shape=circle] (46) at (1,7) {46};
\node[draw,shape=circle] (47) at (2,8) {47};
\node[draw,shape=circle] (56) at (0,8) {56};
\node[draw,shape=circle] (57) at (1,9) {57};
\node[draw,shape=circle] (67) at (0,10) {67};
\draw[line width=3] (12)--(13)--(14)--(15)--(25)--(26)--(36)--(46)--(47)--(57)--(67);
\draw (23)--(24)--(25)--(26);
\draw (34)--(35)--(36);
\draw (45)--(46)--(47);
\draw (56)--(57);
\draw (47)--(57)--(67);
\draw (13)--(23);
\draw (14)--(24)--(34);
\draw (25)--(35)--(45);
\draw (26)--(36)--(46)--(56);
\end{scope}
\end{tikzpicture}    
    \caption{The Hasse diagram of a perfect compatible sublattice of~\(L_7\). The fundamental chain is highlighted in bold.}
    \label{fig: fundamental_chain}
\end{figure}
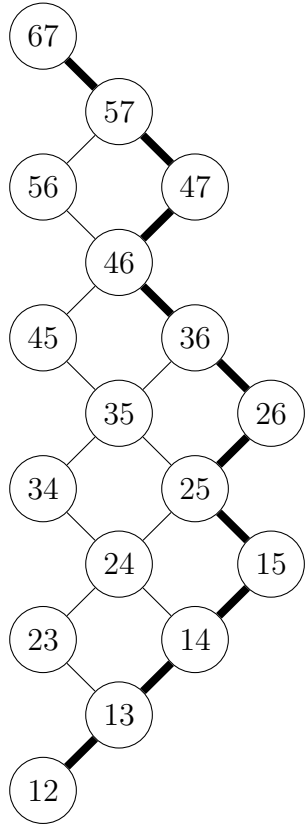

Summarizing the discussion so far, we arrive at the main result of Section~\ref{sec: 2}.

\begin{Theorem}
\label{compatible}
Let \(L\) be a compatible sublattice of \(L_n\). Then \(\mathcal{R}_{\mathbb{K}}[L]\) is an ASL on \(L\) over \(\mathbb{K}\). In particular, \(I_L\) is quadratically generated and \(\mathcal{R}_{\mathbb{K}}[L]\) is Cohen--Macaulay. Furthermore, \(I_L\) is an elimination ideal of the Plücker ideal \(I_{L_n}\).
\end{Theorem}

\begin{proof}
Let $p_{i \ell}$ and $p_{j k}$ with $i < j < k < \ell$ belong to $L$.  Since $L_n \setminus L$ is a poset ideal of $\Pi_n$, both $p_{ij}$ and $p_{k\ell}$ belong to $L$.  It then follows from Lemma \ref{sublattice}  that $\Rc_\mathbb{K}[L]$ is an ASL on $L$ over $\mathbb{K}$.  In particular,
\begin{eqnarray}
\label{QQQ}    
I_{L} = (Q_{ijk\ell} : 1\leq i<j<k<\ell \leq n, \; p_{i\ell} \in L).
\end{eqnarray}
Since every sublattice of a distributive lattice is distributive and since a distributive lattice is wonderful \cite[p.~263]{Eisenbud80}, it follows from \cite[Corollary 4.2]{Eisenbud80} that $\Rc_\mathbb{K}[L]$ is Cohen--Macaulay.

Fix the lexicographic order $<_{\lex}$ on the polynomial ring $S=\mathbb{K}[p_{ij} : 1 \leq i < j \leq n]$ induced by the variable order such that
\begin{itemize}
    \item[$(i)$] $p_{k\ell} <_{\lex} p_{ij}$ if $p_{ij} < p_{k \ell}$ in $\Pi_n$;
    \item[$(ii)$] $p_{k\ell} <_{\lex} p_{ij}$ if $p_{ij} \not\in L$ and $p_{k\ell} \in L$.
\end{itemize}
 It follows from (\ref{QQQ}) that  
\begin{eqnarray}
\label{QQQQQ}
I_{L} = (Q_{ijk\ell} : 1\leq i < j < k < \ell \leq n) \cap \mathbb{K}[\,p_{i\ell} : p_{i\ell} \in L].
\end{eqnarray}
Since \(p_{k\ell} <_{\lex} p_{ij}\) whenever \(p_{ij} \notin L\) and \(p_{k\ell} \in L\), and since the set of quadratic Plücker relations \eqref{quadratic} forms a Gröbner basis of \(I_{L_n}\) with respect to \(<_{\lex}\) (Lemma~\ref{lex}), it follows from \eqref{QQQQQ} that \(I_L\) is an elimination ideal of \(I_{L_n}\), as desired.
\end{proof}

We conclude this section with several further definitions and examples. 
\begin{Definition}
 A compatible sublattice $L$ of $L_n$ is called {\em perfect} if 
 $$\dim \Rc_\mathbb{K}[L] = 2n-3.$$   
\end{Definition}
Since 
$\dim \Rc_\mathbb{K}[L] = \rank(L) + 1 \leq \rank(L_n) + 1 = 2n-3$, it follows that \(L\) is perfect if and only if \(p_{i,i+1} \in L\) for each \(1 \leq i \leq n-1\) and \(p_{i,i+2} \in L\) for each \(1 \leq i \leq n-2\).  The number of perfect compatible sublattices of \(L_n\) is the Catalan number (for other combinatorial interpretations see \cite{Stanley15}):
\[
 C_{n-2} = \frac{1}{n-1}\binom{2n - 4}{\,n - 2\,}.
 \]
\begin{Example} For $n=5$ there are $C_3 = 5$ perfect compatible sublattices of $L_5$. Their Hasse diagrams are shown in Figure~\ref{fig: perfect_sublattices_L5}.
    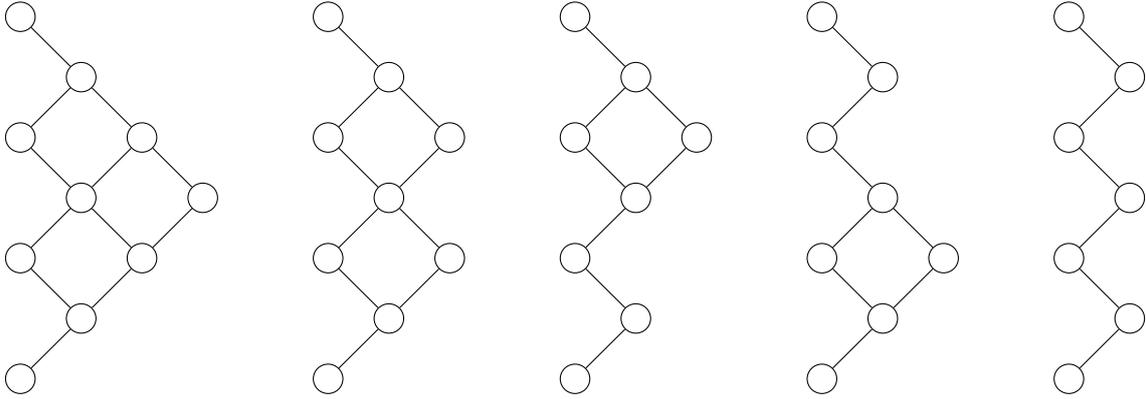
\begin{figure}
    \centering
\begin{tikzpicture}[scale=0.8]
\node[draw,shape=circle] (14) at (4,2) {};
\node[draw,shape=circle] (15) at (5,3) {};
\node[draw,shape=circle] (16) at (6,4) {};
\node[draw,shape=circle] (17) at (7,5) {};
\node[draw,shape=circle] (25) at (4,4) {};
\node[draw,shape=circle] (26) at (5,5) {};
\node[draw,shape=circle] (27) at (6,6) {};
\node[draw,shape=circle] (36) at (4,6) {};
\node[draw,shape=circle] (37) at (5,7) {};
\node[draw,shape=circle] (47) at (4,8) {};
\draw (14)--(15)--(16)--(17);
\draw (25)--(26)--(27);
\draw (36)--(37);
\draw (17)--(27)--(37)--(47);
\draw (15)--(25);
\draw(16)--(26)--(36);
\end{tikzpicture} \ \ \ \ \ \ \ \ \ 
\begin{tikzpicture}[scale=0.8]
\node[draw,shape=circle] (14) at (4,2) {};
\node[draw,shape=circle] (15) at (5,3) {};
\node[draw,shape=circle] (16) at (6,4) {};
\node[draw,shape=circle] (25) at (4,4) {};
\node[draw,shape=circle] (26) at (5,5) {};
\node[draw,shape=circle] (27) at (6,6) {};
\node[draw,shape=circle] (36) at (4,6) {};
\node[draw,shape=circle] (37) at (5,7) {};
\node[draw,shape=circle] (47) at (4,8) {};
\draw (14)--(15)--(16);
\draw (25)--(26)--(27);
\draw (36)--(37);
\draw (27)--(37)--(47);
\draw (15)--(25);
\draw(16)--(26)--(36);
\end{tikzpicture} \ \ \ \ \ \ \ \ \ 
\begin{tikzpicture}[scale=0.8]
\node[draw,shape=circle] (14) at (4,2) {};
\node[draw,shape=circle] (15) at (5,3) {};
\node[draw,shape=circle] (25) at (4,4) {};
\node[draw,shape=circle] (26) at (5,5) {};
\node[draw,shape=circle] (27) at (6,6) {};
\node[draw,shape=circle] (36) at (4,6) {};
\node[draw,shape=circle] (37) at (5,7) {};
\node[draw,shape=circle] (47) at (4,8) {};
\draw (14)--(15);
\draw (25)--(26)--(27);
\draw (36)--(37);
\draw (27)--(37)--(47);
\draw (15)--(25);
\draw (26)--(36);
\end{tikzpicture} \ \ \ \ \ \ \ \ \ 
\begin{tikzpicture}[scale=0.8]
\node[draw,shape=circle] (14) at (4,2) {};
\node[draw,shape=circle] (15) at (5,3) {};
\node[draw,shape=circle] (16) at (6,4) {};
\node[draw,shape=circle] (25) at (4,4) {};
\node[draw,shape=circle] (26) at (5,5) {};
\node[draw,shape=circle] (36) at (4,6) {};
\node[draw,shape=circle] (37) at (5,7) {};
\node[draw,shape=circle] (47) at (4,8) {};
\draw (14)--(15)--(16);
\draw (25)--(26);
\draw (36)--(37);
\draw (37)--(47);
\draw (15)--(25);
\draw (16)--(26)--(36);
\end{tikzpicture} \ \ \ \ \ \ \ \ \ 
\begin{tikzpicture}[scale=0.8]
\node[draw,shape=circle] (14) at (4,2) {};
\node[draw,shape=circle] (15) at (5,3) {};
\node[draw,shape=circle] (25) at (4,4) {};
\node[draw,shape=circle] (26) at (5,5) {};
\node[draw,shape=circle] (36) at (4,6) {};
\node[draw,shape=circle] (37) at (5,7) {};
\node[draw,shape=circle] (47) at (4,8) {};
\draw (14)--(15)--(25)--(26)--(36)--(37)--(47);
\end{tikzpicture}
    \caption{Perfect compatible sublattices of the lattice $L_5$.}
    \label{fig: perfect_sublattices_L5}
\end{figure}
\end{Example}
Let \(L\) be a perfect compatible sublattice of \(L_n\). Since \(L_n \setminus L\) is a poset ideal of \(\Pi_n\), it follows that if \(p_{ij} \in L\), then \(p_{i+1,j} \in L\).  
For \(p_{ij} \in L\), we define \(p_{ij}^{(1)}\) to be \(p_{i,j+1}\) if \(p_{i,j+1} \in L\), and \(p_{i+1,j}\) otherwise. Recursively, for \(k \geq 2\), we set
\[
p_{ij}^{(k)} := \bigl(p_{ij}^{(k-1)}\bigr)^{(1)}.
\]
The chain
\[
p_{12} < p_{12}^{(1)} < p_{12}^{(2)} < \cdots < p_{12}^{(2n-4)}
 \]
in \(L\) of length \(2n-4\) is called the \emph{fundamental chain}.
\begin{Example}
    The fundamental chain of a perfect compatible sublattice of $L_7$ is shown in Figure~\ref{fig: fundamental_chain}.
\end{Example}

\section{Interval graphs and the Gorenstein property}  \label{sec: 3}
Let $G$ be a finite simple graph on $[n]=\{1, \ldots, n\}$ with no isolated vertices. 
A \emph{clique} of $G$ is a subset $C \subseteq [n]$ such that, for any distinct 
$i,j \in C$, the pair $\{i,j\}$ is an edge of $G$. 
Let $\Delta(G)$ denote the set of cliques of $G$. 
Since $\Delta(G)$ is a simplicial complex on $[n]$, we call $\Delta(G)$ the 
\emph{clique complex} of $G$. 
A \emph{maximal clique} of $G$ is a maximal face (i.e.~a facet) of $\Delta(G)$.

A \emph{chordal} graph is a finite graph $G$ in which every cycle of length greater 
than $3$ has a chord. We review a distinguished subfamily of chordal graphs.
\begin{Definition}
  We say that $G$ is an \emph{interval graph} if each maximal clique of $G$ is an 
interval $[a,b]$ with $1 \leq a < b \leq n$, where
$
[a,b] = \{a, a+1, \ldots, b\}.
$  
\end{Definition}

\begin{Example}
The finite graph on the left side of Figure~\ref{fig: interval_graph} is an interval graph. In contrast, the isomorphic finite graph on the right side of Figure~\ref{fig: interval_graph} is not an interval graph: a clique $\{2,4,5\}$ is not an interval. 
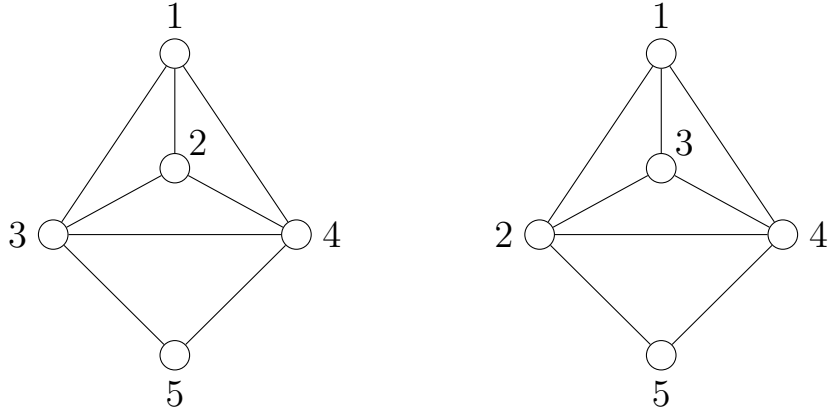
\begin{figure}
    \centering
\begin{tikzpicture}[scale=0.8]
\node[draw,shape=circle] (1) at (2,3) {};
\node[draw,shape=circle] (2) at (2,1.1) {};
\node[draw,shape=circle] (3) at (0,0) {};
\node[draw,shape=circle] (4) at (4,0) {};
\node[draw,shape=circle] (5) at (2,-2) {};
\node[draw,shape=circle] (6) at (10,3) {};
\node[draw,shape=circle] (7) at (10,1.1) {};
\node[draw,shape=circle] (8) at (8,0) {};
\node[draw,shape=circle] (9) at (12,0) {};
\node[draw,shape=circle] (10) at (10,-2) {};
\draw (1)node[above=2mm]{\large{1}}--(2)node[above right=0.5mm]{\large{2}}--(3)node[left=2mm]{\large{3}}--(4)node[right=2mm]{\large{4}}--(5)node[below=2mm]{\large{5}}--(3)--(1)--(4);
\draw (2)--(4);
\draw (6)node[above=2mm]{\large{1}}--(7)node[above right=0.5mm]{\large{3}}--(8)node[left=2mm]{\large{2}}--(9)node[right=2mm]{\large{4}}--(10)node[below=2mm]{\large{5}}--(8)--(6)--(9);
\draw (7)--(9);
\end{tikzpicture}
    \caption{Two isomorphic graphs: an interval graph (left) and a non-interval graph (right).}
    \label{fig: interval_graph}
\end{figure}
\end{Example}

\begin{Lemma}
Every interval graph is a chordal graph.   
\end{Lemma}

\begin{proof}
Let $G$ is an interval graph on $[n]$ and suppose that the maximal cliques of $G$ are $C_1 =[a_1, b_1], \ldots, C_s=[a_s, b_s]$ with $1=a_1 < \cdots < a_s < n$.  Since 
\[
C_i \cap C_j \subset C_{j-1} \cap C_j, \quad \text{for } \; 1 \leq i < j \leq s, 
\]
the ordering $C_1, \ldots, C_s$ is a \emph{leaf order} \cite[p.~172]{HHgtm260} of $\Delta(G)$.  Hence $\Delta(G)$ is a \emph{quasi-forest} \cite[p.~172]{HHgtm260} and $G$ is a chordal graph (\cite[Theorem 9.2.12]{HHgtm260}).  
\end{proof}

Let, as in the previous section, \(L_n = \{p_{ij} : 1 \leq i < j \leq n\}\) be the finite distributive lattice. Given an arbitrary subposet \(L\) of \(L_n\), we associate to \(L\) a finite graph \(G_L\) on \([n]\) whose edges are precisely those \(\{i,j\}\) for which \(p_{ij} \in L\).

\begin{Example}
    Let $L$ be the subposet of $L_6$ illustrated in Figure~\ref{fig: subposet_graph} on the left.  Then $G_L$ is the finite graph on $[6]$ shown in the same figure on the right. 
\end{Example}
\begin{Lemma}
    \label{Sydney}
    Let $L$ be a subposet of $L_n$.  Then $L_n \setminus L$ is a poset ideal of $\Pi_n$ if and only if $G_L$ is an interval graph.  
\end{Lemma}

\begin{proof}
By virtue of \cite[Theorem~2.2]{EHH11}, a finite graph $G$ on $[n]$ is an interval graph if and only if the following condition is satisfied:

\medskip
\noindent$(*)$
If $\{i,j\}$ and $\{k,\ell\}$ are edges of $G$ with $i<j$ and $k<\ell$, then
\begin{itemize}
  \item $\{j,\ell\}$ is an edge of $G$ if $i=k$ and $j\neq \ell$, and
  \item $\{i,k\}$ is an edge of $G$ if $i\neq k$ and $j=\ell$.
\end{itemize}

First, suppose that $L_n \setminus L$ is a poset ideal of $\Pi_n$. 
Let $\{i,j\}$ and $\{k,\ell\}$, with $i<j$ and $k<\ell$, be edges of $G_L$.
If $i=k$ and $j<\ell$, then, since $L_n \setminus L$ is a poset ideal of $\Pi_n$, 
it follows that $\{k',\ell\}$ is an edge of $G_L$ for all $k' \geq k$.
In particular, $\{j,\ell\}$ is an edge of $G_L$.
If $i<k$ and $j=\ell$, then, again because $L_n \setminus L$ is a poset ideal of $\Pi_n$, 
it follows that $\{i,j'\}$ is an edge of $G_L$ for all $j' \leq j$.
In particular, $\{i,k\}$ is an edge of $G_L$.
Hence condition $(*)$ is satisfied, and $G_L$ is an interval graph.

Conversely, suppose that the finite graph $G_L$ on $[n]$ is an interval graph whose maximal cliques are
$
C_1=[a_1,b_1], \ldots, C_s=[a_s,b_s],
$
with $1=a_1<\cdots<a_s<n$.
Let $\{i,j\}$ with $i<j$ and $\{k,\ell\}$ with $k<\ell$ satisfy $i\geq k$ and $j\leq \ell$.
Suppose that $\{k,\ell\}$ is an edge of $G_L$, say $\{k,\ell\}\in C_t$.
Then $a_t \leq k < \ell \leq b_t$.
Since $i\geq k$ and $j\leq \ell$, we have
$
a_t \leq i < j \leq b_t,
$
and hence $\{i,j\} \in C_t$.
Thus $\{i,j\}$ is an edge of $G_L$.
This shows that $L_n \setminus L$ is a poset ideal of $\Pi_n$, as desired.
\end{proof}

\begin{Corollary}
    \label{distributive}    
    A subposet $L$ of $L_n$ is a compatible sublattice of $L_n$ if and only if $G_L$ is an interval graph. 
\end{Corollary}

\begin{proof}
{\bf (``only if'')} This is a direct consequence of Lemma \ref{distributive}. 

{\bf (``if'')} Suppose that \(G_L\) is an interval graph. Let \(\{i,\ell\}\) and \(\{j,k\}\) be edges of \(G_L\) with \(i<j<k<\ell\). Since \(G_L\) is an interval graph, it follows that \(\{i,k\}\) and \(\{j,\ell\}\) are also edges of \(G_L\). Hence both \(p_{i\ell} \wedge p_{jk}\) and \(p_{i\ell} \vee p_{jk}\) belong to \(L\). In other words, \(L\) is a sublattice of \(L_n\). It then again follows from Lemma~\ref{distributive} that \(L\) is a compatible sublattice of \(L_n\).
\end{proof}

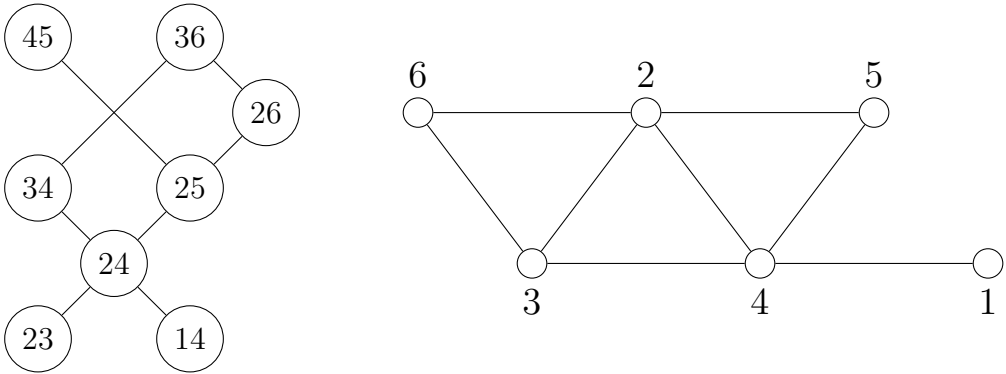
\begin{figure}
    \centering
\begin{tikzpicture}
\node[draw,shape=circle] (14) at (2,2) {14};
\node[draw,shape=circle] (23) at (0,2) {23};
\node[draw,shape=circle] (24) at (1,3) {24};
\node[draw,shape=circle] (25) at (2,4) {25};
\node[draw,shape=circle] (26) at (3,5) {26};
\node[draw,shape=circle] (34) at (0,4) {34};
\node[draw,shape=circle] (36) at (2,6) {36};
\node[draw,shape=circle] (45) at (0,6) {45};
\node[draw,shape=circle] (1) at (12.5,3) {};
\node[draw,shape=circle] (2) at (8,5) {};
\node[draw,shape=circle] (3) at (6.5,3) {};
\node[draw,shape=circle] (4) at (9.5,3) {};
\node[draw,shape=circle] (5) at (11,5) {};
\node[draw,shape=circle] (6) at (5,5) {};
\draw (45)--(25);
\draw (34)--(36)--(26)--(25)--(24)--(23);
\draw (34)--(24)--(14);
\draw (6)node[above=2mm]{\large{6}}--(2)node[above=2mm]{\large{2}}--(5)node[above=2mm]{\large{5}}--(4)node[below=2mm]{\large{4}}--(2)--(3)node[below=2mm]{\large{3}}--(6);
\draw (3)--(4)--(1)node[below=2mm]{\large{1}};
\end{tikzpicture}
    \caption{A subposet $L$ of $L_6$ (left) and its finite graph $G_{L}$ (right).}
    \label{fig: subposet_graph}
\end{figure}

\begin{Lemma}
    \label{LEMMAperfect}
    Let $L$ be a sublattice of $L_n$ and suppose that $G_L$ is an interval graph whose maximal cliques are $C_1 =[a_1, b_1], \ldots, C_s=[a_s, b_s]$ with $1=a_1 < \cdots < a_s < n$.  Then the compatible sublattice $L$ of $L_n$ is perfect if and only if 
    $$|C_i \cap C_{i+1}| \geq 2 \quad \text{for each } \; 1 \leq i < s.$$ 
\end{Lemma}

\begin{proof}
Suppose that $|C_i \cap C_{i+1}| \geq 2$ for each $1 \leq i < s$.
Since $C_i \cap C_{i+1} \neq \emptyset$ for all $1 \leq i < s$, it follows that
$\{j,j+1\}$ is an edge of $G$ for each $1 \leq j < n$.
Fix $1 \leq j \leq n-2$, and choose the largest index $i$ such that $j \in C_i$.
Since $|C_i \cap C_{i+1}| \geq 2$, we have $a_i \leq j \leq b_i - 2$.
Hence $j+2 \in C_i$, and therefore $\{j,j+2\}$ is an edge of $G$.

On the other hand, if there exists an index $1 \leq i_0 < s$ such that
$|C_{i_0} \cap C_{i_0+1}| \leq 1$, then
$\{b_{i_0}-1,\, b_{i_0}+1\}$ cannot be an edge of $G$.
Consequently, $L$ cannot be perfect.
\end{proof}

\begin{Example}
Let $n=5$.
The interval graphs on $[5] = \{ 1,2,3,4,5\}$ satisfying the condition of
Lemma~\ref{LEMMAperfect} are
$$
\{[1,5]\},\ \{[1,4],[2,5]\},\ \{[1,4],[3,5]\},\ \{[1,3],[2,5]\},\ 
\{[1,3],[2,4],[3,5]\},
$$
where, for example, $\{[1,4],[3,5]\}$ denotes the interval graph on $[5]$
whose maximal cliques are $[1,4]$ and $[3,5]$.
\end{Example}

\begin{Remark}
Following the last paragraph of the previous section, the number of perfect compatible sublattices of \(L_n\) is the Catalan number \(C_{n-2}\). It then follows that the number of interval graphs \(G\) satisfying the condition of Lemma~\ref{LEMMAperfect} is also the Catalan number \(C_{n-2}\).
\end{Remark}

We now turn to the discussion of the condition for \(\mathcal{R}_{\mathbb{K}}[L]\) to be Gorenstein. Recall that an element \(a\) of a finite distributive lattice \(L\) is called \emph{join‑irreducible} if, whenever \(a = b \vee c\) with \(b, c \in L\), one has \(a = b\) or \(a = c\). The unique minimal element of \(L\) cannot be join‑irreducible.
\begin{Example}
 The join-irreducible elements of the finite distributive lattice~${L \subset L_7}$ illustrated in
Figure~\ref{fig: fundamental_chain} are
$$
p_{13},\, p_{14},\, p_{15},\, p_{26},\, p_{47},\, p_{23},\, p_{34},\, p_{45},\, p_{56},\, p_{67}.
$$  
\end{Example}

\begin{Lemma}
    \label{Gor}
    Let $L$ be a sublattice of $L_n$ and suppose that $G_L$ is an interval graph whose maximal cliques are $C_1 =[a_1, b_1], \ldots, C_s=[a_s, b_s]$ with $1=a_1 < \cdots < a_s < n$.  
Let $P$ be the subposet of $L$ consisting of all join-irreducible elements of $L$.  Then $P$ is pure if and only if $|C_i \cap C_{i+1}| \leq 3$ for each $1 \leq i < s$.  
\end{Lemma}

\begin{proof}
One has 
$
L=\{p_{ij} : i<j, \, i\in C_q, \, j \in C_q, \, 1 \leq q \leq s \}.  
$

\smallskip

\noindent
{\bf (``if'')} Suppose that $|C_i \cap C_{i+1}| \leq 3$ for each $1 \leq i < s$.
We show, by induction on $s$, that the subposet $P$ of $L$ consisting
of all join-irreducible elements of $L$ is pure.

Let $s = 1$. Then the join-irreducible elements of $L$ are
$p_{1j}$ with $3 \leq j \leq b_1$ and $p_{j,j+1}$ with
$2 \leq j < b_1$.
Since the partial order of $P$ is
\begin{eqnarray*}
&p_{12} < p_{13} < p_{23} < \cdots < p_{b_1-2, b_1 - 1} < p_{b_1-1,b_1},&\\
&p_{12} < p_{13} < p_{14} < \cdots < p_{1b_1} < p_{b_1-1,b_1},&\\
& p_{1j} < p_{j-1,j},
\end{eqnarray*}
for $4 \leq j \leq b_1$, it follows that the poset $P$ is pure.

Let $s > 2$, and let $G'$ be the subgraph of $G_L$ whose maximal cliques
are $C_1, \ldots, C_{s-1}$. We denote by $L'$ a sublattice of $L$ corresponding to the graph $G'$.
Suppose that the subposet $P'$ of $L'$ consisting of all
join-irreducible elements of $L'$ is pure.
The maxi\-mal element of $P'$ is $p_{b_{s-1}-1, b_{s-1}}$.
The poset $P$ is obtained from $P'$ by adding $p_{a_s, j}$ with
$b_{s-1} < j \leq b_s$ and $p_{j,j+1}$ with
$b_{s-1} \leq j < b_s$.
The partial order of $P\setminus P'$ is
\begin{eqnarray*}
&p_{a_s, b_{s-1}+1} < p_{a_s, b_{s-1}+2} < \cdots < p_{a_s, b_s},&\\
&p_{b_{s-1}, b_{s-1}+1} < p_{b_{s-1}+1, b_{s-1}+2}
< \cdots < p_{b_s-1,b_s},&\\
& p_{a_s,j} < p_{j-1,j}, 
\end{eqnarray*}
for $b_{s-1} < j \leq b_s$, which is pure.
Let $|C_{s-1} \cap C_s| = 3$. Then
\begin{eqnarray*}
&p_{a_s, a_s+1} < p_{b_{s-1}-1,b_{s-1}} < p_{b_{s-1}, b_{s-1}+1},&\\
&p_{a_s, a_s+1} < p_{a_s, b_{s-1}+1}.&
\end{eqnarray*}
Now the induction hypothesis guarantees that $P$ is pure.
A similar argument shows that $P$ is pure if
$|C_{s-1} \cap C_s| < 3$.

\smallskip

\noindent
{\bf (``only if'')} Suppose that $|C_{s-1} \cap C_s| \geq 4$.
Let
$C_{s-1} = \{\ldots, a-1, a, a+1, \ldots, b\}$ and
$C_s = \{a, a+1, \ldots, b, \ldots, n\}$ with $b - a \geq 3$.
Then the interval $[p_{a,a+1},\, p_{b,b+1}]$ is 
\begin{eqnarray*}
&p_{a,a+1} < p_{a+1,a+2} < p_{a+2,a+3} < \cdots < p_{b,b+1},&\\
&p_{a,a+1} < p_{a,b+1} < p_{b,b+1},&
\end{eqnarray*}
which is not pure, since $b - a \geq 3$.
Thus $P$ cannot be pure, as required.
\end{proof}

We now turn to the extension of Theorem~\ref{compatible}.

\begin{Theorem}
    \label{Th_Gorenstein}
A sublattice $L$ of $L_n$ is a compatible sublattice of $L_n$ if and only if $G_L$ is an interval graph.  Furthermore, if $G_L$ is an interval graph whose maximal cliques are $C_1 =[a_1, b_1], \ldots, C_s=[a_s, b_s]$ with $1=a_1 < \cdots < a_s < n$, then $\Rc_K(L)$ is Gorenstein if and only if $|C_i \cap C_{i+1}| \leq 3$ for each $1 \leq i < s$. 
\end{Theorem}

\begin{proof}
Since $\mathcal{R}(G)=S/I_G$ is a Cohen--Macaulay integral domain (Theorem \ref{compatible}), by virtue of \cite[Theorem 4.4]{Stanley78}, the criterion on \cite[p.~107]{Hibi87} applies to $\mathcal{R}(G)=S/I_G$.  In other words, $\mathcal{R}(G)=S/I_G$ is Gorenstein if and only if the subposet $P$ of $L$ which consists of all join-irreducible elements of $L$ is pure.  It then follows from Lemma \ref{Gor} that $\mathcal{R}(G)=S/I_G$ is Gorenstein if and only if $|C_i \cap C_{i+1}| \leq 3$ for each $1 \leq i < s$. 
\end{proof}

\begin{Corollary}
    \label{Cor_Gorenstein}
A sublattice $L$ of $L_n$ is a perfect compatible sublattice of $L_n$ for which $\Rc_K(L)$ is Gorenstein if and only if $G_L$ is an interval graph whose maximal cliques $C_1 =[a_1, b_1], \ldots, C_s=[a_s, b_s]$ with $1=a_1 < \cdots < a_s < n$ satisfy 
$$2 \leq |C_i \cap C_{i+1}| \leq 3 \quad \text{for each } \; 1 \leq i < s.$$ 
\end{Corollary}

Finally, we enumerate the number of Gorenstein rings
$\mathcal{R}_{\mathbb{K}}[L]$ with
$\dim \mathcal{R}_{\mathbb{K}}[L] = 2n - 3$, where $L$ is a perfect
compatible sublattice of $L_n$.
First, we define the sequences $\{p_n\}_{n=0}^\infty$ and
$\{q_n\}_{n=0}^\infty$ by the recurrence relations
\begin{eqnarray}
\label{recurrence}
p_{n+1} = 2p_n + q_n, \quad
q_{n+1} = p_n + q_n, \quad n = 0,1,2,3,\ldots
\end{eqnarray}
together with the initial conditions $p_0 = 1$ and $q_0 = 1$.

\begin{Lemma}
  \label{enumeration}
  Let $n \geq 4$.  The number of interval graphs on $[n] = \{1,2,\dots,n\}$ satisfying the condition of Corollary \ref{Cor_Gorenstein} is $p_{n-4}+q_{n-4}$.
\end{Lemma}

\begin{proof}
Let $n \geq 0$.
Let $p_n$ (respectively $q_n$) denote the number of interval graphs $G$ on
$[n+4]$ satisfying the condition of Corollary~\ref{Cor_Gorenstein} for which
$|C| \geq 4$ (respectively $|C| = 3$), where $C$ is the maximal clique of $G$
with $n \in C$.
In particular, the number of interval graphs on $[n]$ satisfying the
condition of Corollary~\ref{Cor_Gorenstein} is $p_{n-4} + q_{n-4}$.
Now, one has $p_0 = q_0 = 1$, and it follows easily that the sequences
$\{p_n\}_{n=0}^\infty$ and $\{q_n\}_{n=0}^\infty$ satisfy the recurrence
formula~\eqref{recurrence}.
\end{proof}

Finally, by expressing $p_n$ and $q_n$ in terms of $n$, one obtains the following.
\begin{Corollary}
  \label{enumeration_2}
  Let $n \geq 4$.  The number of Gorenstein rings $\Rc_\mathbb{K}[L]$ with $\dim \Rc_\mathbb{K}[L] = 2n - 3$, where $L$ is a perfect compatible sublattice of $L_n$, is 
  \[
  \frac{1}{\,5\,}
  \left[
(5+2\sqrt{5}\,)\left(\frac{\,3+\sqrt{5}\,}{2}\right)^{n-4}
  +\,
(5-2\sqrt{5}\,)\left(\frac{\,3-\sqrt{5}\,}{2}\right)^{n-4}\,
  \right].
  \]
\end{Corollary}

  \appendix
\section{Another Gröbner basis for the Plücker ideal} \label{appendix}  
In Section~\ref{sec: 2} we provided a lexicographic Gröbner basis for the Plücker ideal
$I_{L_n}$ associated with the finite poset $\Pi_n$.
In this appendix, we construct another lexicographic Gröbner basis for a
natural lexicographic ordering of the Plücker variables.
Before stating this result, we first prove a combinatorial lemma.

\smallskip

Richard Stanley recorded 214 combinatorial interpretations of the Catalan numbers in his monograph \cite{Stanley15}. The following lemma adds one more to this list.
\begin{Lemma}\label{lem: arcs_with_forbidden_patterns}
Let arcs $(a,b)$ be drawn on $n$ points $1,\dots,n$ placed on a horizontal line. Suppose the arrangement avoids the following patterns:
\begin{itemize}
  \item two disjoint arcs $(i,j)$ and $(k,l)$ with $1 \le i<j<k<l \le n$;
  \item three arcs $(i,k),(j,l),(k,m)$ with $1 \le i<j<k<l<m \le n$;
  \item three arcs $(i,l),(j,m),(k,s)$ with $1 \le i<j<k<l<m<s \le n$.
\end{itemize}
Then:
\begin{enumerate}
  \item every non-maximal arrangement extends to one with $2n-3$ arcs;
    \item the maximum number of arcs in such an arrangement is $2n-3$;
  \item the number of maximal arrangements equals the Catalan number $C_{n-2}$.
\end{enumerate}
\end{Lemma}
\begin{proof}
By the first forbidden pattern, no two arcs are disjoint, hence any two distinct arcs $(a,b)$ and $(c,d)$ either intersect $(a<c\leq b < d)$ or are nested $(a\leq c <  d \leq b)$.  

Define the length of an arc \((a,b)\) to be \(b-a\).
There is at most one arc of length \(n-1\), namely \((1,n)\).  
For any \(1 \le \ell \le n-2\), there are at most two arcs of length \(\ell\). Indeed, three distinct arcs of the same length cannot be nested and therefore must pairwise intersect, which necessarily creates one of the forbidden triples.
Thus the total number of arcs is at most $1 + 2(n-2) = 2n-3$, which proves~$(2)$. This bound is attained, for example, by the following arrangement 
$$\{(1,k), (2,l) \mid k \in \{2,\dots,n\}, \; l \in \{3,\dots,n\}\}.$$

We now prove (1) and (3) by induction on \(n\).  
For \(n=2\), the only allowed arrangements are the empty one and the one consisting of a single arc, so both statements hold. To prove (3), we construct a bijection with full binary trees with \(n-1\) leaves, equivalently, with \(n-2\) internal nodes. Their number is well-known to be the Catalan number $C_{n-2}$, see \cite{Stanley15}. For instance, for \(n=2\) the tree consists of a single node corresponding to the arc \((1,2)\).

Now let \(n\ge 3\). Observe that the arcs \((1,n)\), \((1,n-1)\), and \((2,n)\) can always be added to any allowed arrangement without creating a forbidden pattern. Thus, it suffices to consider arrangements on \(n\) points containing these three arcs.

Suppose first that all remaining arcs are nested inside a single arc of length \(n-2\), say \((1,n-1)\). Removing the vertex \(n\) together with the arcs \((1,n)\) and \((2,n)\) yields an allowed arrangement on \(n-1\) points, which by induction extends to an arrangement $\mathcal{A}'$ with \(2n-5\) arcs. Adding back the arcs \((1,n)\) and \((2,n)\) produces an allowed arrangement \(\mathcal A\) on \(n\) points with \(2n-3\) arcs, proving (1) in this case. By the induction hypothesis, \(\mathcal A'\) corresponds to a full binary tree rooted at the arc \((1,n-1)\). To construct the tree corresponding to \(\mathcal A\), we add a new root labeled by \((1,n)\), whose left child is the tree for \(\mathcal A'\) and whose right child is the leaf corresponding to \((2,n)\).

Now assume that not all arcs of length at most \(n-3\) are nested inside a single arc of length \(n-2\), say \((1,n-1)\). Let \((i_1,n),\dots,(i_k,n)\) be all arcs with right endpoint at~\(n\).
Replace each arc \((i_j,n)\) by \((i_j,n-1)\). If this coincides with an existing arc, we keep a single copy. Removing the vertex \(n\) and all arcs incident to it yields an arrangement on points \(1,\dots,n-1\). This squeezing operation does not create forbidden patterns, since all right endpoints move inward while preserving their relative order. Hence the resulting arrangement is allowed.
By induction, it extends to an allowed arrangement \(\mathcal B'\) with \(2n-5\) arcs. 

Let \((a,n-1)\) be the shortest arc with right endpoint \(n-1\) in \(\mathcal B'\).
We now reinsert the vertex \(n\) and replace every arc \((j,n-1)\) by \((j,n)\), this again avoids the three forbidden patterns. Finally, we add the two arcs \((1,n)\) and \((a,n-1)\). The resulting arrangement \(\mathcal B\) is allowed and has
$2n-3$
arcs, proving (1) in this case.

On the level of trees, \(\mathcal B'\) corresponds by induction to a full binary tree \(T'\) rooted at \((1,n-1)\) with left and right children $(1,n-2)$ and $(2,n-1)$, respectively. Let \((a,n-1)\) be the rightmost leaf of minimal length. Removing the subtree rooted at \((2,n-1)\) and identifying the right leaf of  \((1,n-1)\) with \((a,n-1)\) produces a tree \(T_1\).  From the removed subtree we form a new tree \(T_2\) by replacing \(n-1\) with \(n\) in every node label. The tree corresponding to \(\mathcal B\) is obtained by introducing a new root labeled \((1,n)\) with left child \(T_1\) and right child \(T_2\).

Finally, observe that for every internal node labeled by an arc \((a,b)\), its two children are uniquely determined by the maximal arcs of the form \((a,j)\) and \((i,b)\)  with $j \geq i$ nested inside it. The forbidden patterns guarantee both existence and uniqueness of this choice, ensuring that the construction indeed produces a full binary tree.
\end{proof}    
\begin{Example}
For $n=5$, there are exactly $C_3 = 5$ nested arc arrangements avoiding the
patterns from Lemma~\ref{lem: arcs_with_forbidden_patterns}.
The diagrams of these arc arrangements are shown in the first row of
Figure~\ref{fig: arcs}.
The corresponding full binary trees with four leaves are shown in the second
row of Figure~\ref{fig: arcs}.
Note that the first and the last two binary trees are obtained from a full
binary tree with three leaves by adding a new root together with an additional
child.
This corresponds to the first case in the induction step of the proof of
Lemma~\ref{lem: arcs_with_forbidden_patterns}.
The middle binary tree is obtained by the procedure described in the second
case of the proof.
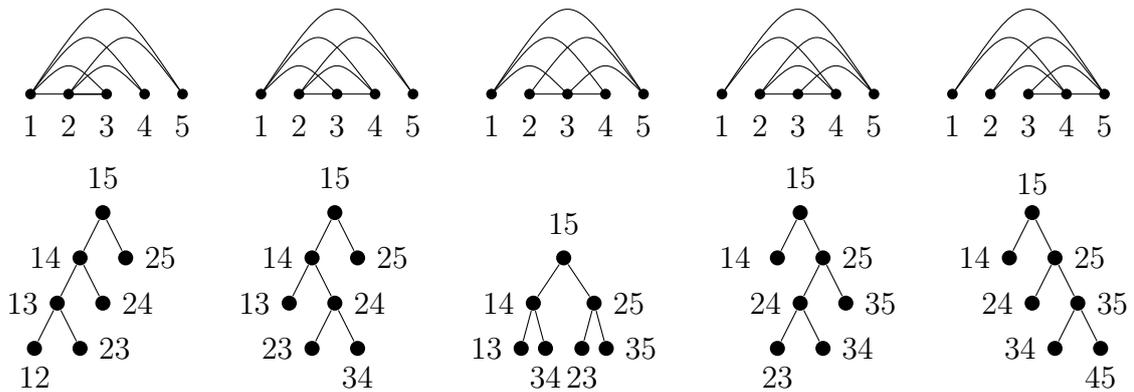
\begin{figure}[h]
\centering
\usetikzlibrary{arrows.meta,positioning}

\tikzset{
  treenode/.style = {circle, draw, minimum size=7mm, inner sep=1pt},
  level 1/.style={sibling distance=30mm},
  level 2/.style={sibling distance=18mm},
  level 3/.style={sibling distance=10mm},
}

\begin{tabular}{c c}

\begin{tikzpicture}[scale=0.5]
  \foreach \x in {1,...,5} {
    \fill (\x,0) circle (4pt);
    \node[below=4pt] at (\x,0) {$\x$};
  }

  \draw (2,0) .. controls (2.5,0) .. (3,0);
  \draw (1,0) .. controls (3.5,0) .. (2,0);
  \draw (1,0) .. controls (2.5,2.0) .. (4,0);
  \draw (2,0) .. controls (3.5,2.0) .. (5,0);
  \draw (1,0) .. controls (2,1.0) .. (3,0);
  \draw (2,0) .. controls (3,1.0) .. (4,0);
  \draw (1,0) .. controls (3,3.0) .. (5,0);
\end{tikzpicture}
\quad

\begin{tikzpicture}[scale=0.5]
  \foreach \x in {1,...,5} {
    \fill (\x,0) circle (4pt);
    \node[below=4pt] at (\x,0) {$\x$};
  }

  \draw (2,0) .. controls (2.5,0) .. (3,0);
  \draw (3,0) .. controls (3.5,0) .. (4,0);
  \draw (1,0) .. controls (2.5,2.0) .. (4,0);
  \draw (2,0) .. controls (3.5,2.0) .. (5,0);
  \draw (1,0) .. controls (2,1.0) .. (3,0);
  \draw (2,0) .. controls (3,1.0) .. (4,0);
  \draw (1,0) .. controls (3,3.0) .. (5,0);
\end{tikzpicture}
\quad
\begin{tikzpicture}[scale=0.5]
  \foreach \x in {1,...,5} {
    \fill (\x,0) circle (4pt);
    \node[below=4pt] at (\x,0) {$\x$};
  }

  \draw (2,0) .. controls (2.5,0) .. (3,0);
  \draw (3,0) .. controls (3.5,0) .. (4,0);
  \draw (1,0) .. controls (2.5,2.0) .. (4,0);
  \draw (2,0) .. controls (3.5,2.0) .. (5,0);
  \draw (1,0) .. controls (2,1.0) .. (3,0);
  \draw (3,0) .. controls (4,1.0) .. (5,0);
  \draw (1,0) .. controls (3,3.0) .. (5,0);
\end{tikzpicture}
\quad
\begin{tikzpicture}[scale=0.5]
  \foreach \x in {1,...,5} {
    \fill (\x,0) circle (4pt);
    \node[below=4pt] at (\x,0) {$\x$};
  }

  \draw (2,0) .. controls (2.5,0) .. (3,0);
  \draw (3,0) .. controls (3.5,0) .. (4,0);
  \draw (1,0) .. controls (2.5,2.0) .. (4,0);
  \draw (2,0) .. controls (3.5,2.0) .. (5,0);
  \draw (2,0) .. controls (3,1.0) .. (4,0);
  \draw (3,0) .. controls (4,1.0) .. (5,0);
  \draw (1,0) .. controls (3,3.0) .. (5,0);
\end{tikzpicture}
\quad
\begin{tikzpicture}[scale=0.5]
  \foreach \x in {1,...,5} {
    \fill (\x,0) circle (4pt);
    \node[below=4pt] at (\x,0) {$\x$};
  }

  \draw (4,0) .. controls (4.5,0) .. (5,0);
  \draw (3,0) .. controls (3.5,0) .. (4,0);
  \draw (1,0) .. controls (2.5,2.0) .. (4,0);
  \draw (2,0) .. controls (3.5,2.0) .. (5,0);
  \draw (2,0) .. controls (3,1.0) .. (4,0);
  \draw (3,0) .. controls (4,1.0) .. (5,0);
  \draw (1,0) .. controls (3,3.0) .. (5,0);
\end{tikzpicture}
\\
\begin{tikzpicture}[scale = 0.4,
  treenode/.style = {circle, fill, inner sep=2pt},
  edge from parent/.style = {draw},
  level 1/.style={sibling distance=15mm},
  level 2/.style={sibling distance=15mm},
  level 3/.style={sibling distance=15mm}
]

\node[treenode] (n15) {}
  child {
    node[treenode] (n14) {}
      child {
        node[treenode] (n13) {}
          child { node[treenode] (n12) {} }
          child { node[treenode] (n23) {} }
      }
      child { node[treenode] (n24) {} }
  }
  child { node[treenode] (n25) {} };

 \node[above=5pt] at (n15) {$15$};
\node[left=3pt] at (n14) {$14$};
\node[left=3pt] at (n13) {$13$};
\node[below=3pt] at (n12) {$12$};
\node[right=3pt] at (n23) {$23$};
\node[right=3pt] at (n24) {$24$};
\node[right=3pt] at (n25) {$25$};
\end{tikzpicture}
\quad
\begin{tikzpicture}[scale = 0.4,
  treenode/.style = {circle, fill, inner sep=2pt},
  edge from parent/.style = {draw},
  level 1/.style={sibling distance=15mm},
  level 2/.style={sibling distance=15mm},
  level 3/.style={sibling distance=15mm}
]

\node[treenode] (n15) {}
  child {
    node[treenode] (n14) {} 
      child { node[treenode] (n13) {} }
      child {
        node[treenode] (n24) {}
          child { node[treenode] (n23) {} }
          child { node[treenode] (n34) {} }
      }
  }
  child { node[treenode]{} };

\node[above=5pt] at (n15) {$15$};
\node[left=3pt] at (n14) {$14$};
\node[left=3pt] at (n13) {$13$};
\node[below=3pt] at (n34) {$34$};
\node[left=3pt] at (n23) {$23$};
\node[right=3pt] at (n24) {$24$};
\node[right=3pt] at (n25) {$25$};
\end{tikzpicture}
\quad
\begin{tikzpicture}[scale = 0.4,
  treenode/.style = {circle, fill, inner sep=2pt},
  edge from parent/.style = {draw},
  level 1/.style={sibling distance=20mm},
  level 2/.style={sibling distance=8mm},
  level 3/.style={sibling distance=15mm}
]

\node[treenode] (n15) {}
  child {
    node[treenode] (n14) {}
      child { node[treenode] (n13) {} }
      child { node[treenode] (n34) {} }
  }
  child {
    node[treenode] (n25) {}
      child { node[treenode] (n23) {} }
      child { node[treenode] (n35) {} }
  };

\node[above=5pt] at (n15) {$15$};
\node[left=3pt] at (n14) {$14$};
\node[left=3pt] at (n13) {$13$};
\node[below=3pt] at (n34) {$34$};
\node[below=3pt] at (n23) {$23$};
\node[right=3pt] at (n35) {$35$};
\node[right=3pt] at (n25) {$25$};
\end{tikzpicture}
\quad
\begin{tikzpicture}[scale = 0.4,
  treenode/.style = {circle, fill, inner sep=2pt},
  edge from parent/.style = {draw},
  level 1/.style={sibling distance=15mm},
  level 2/.style={sibling distance=15mm},
  level 3/.style={sibling distance=15mm}
]

\node[treenode] (n15) {}
  child { node[treenode] (14) {} }
  child {
    node[treenode] (n25) {}
      child {
        node[treenode] (n24) {}
          child { node[treenode] (n23) {} }
          child { node[treenode] (n34) {} }
      }
      child { node[treenode] (n35) {} }
  };
\node[above=5pt] at (n15) {$15$};
\node[left=3pt] at (n14) {$14$};
\node[left=3pt] at (n24) {$24$};
\node[right=3pt] at (n34) {$34$};
\node[below=3pt] at (n23) {$23$};
\node[right=3pt] at (n35) {$35$};
\node[right=3pt] at (n25) {$25$};
\end{tikzpicture}
\quad
\begin{tikzpicture}[scale = 0.4,
  treenode/.style = {circle, fill, inner sep=2pt},
  edge from parent/.style = {draw},
  level 1/.style={sibling distance=15mm},
  level 2/.style={sibling distance=15mm},
  level 3/.style={sibling distance=15mm}
]

\node[treenode] (n15) {}
child { node[treenode] (n14) {} }
  child {
    node[treenode] (n25) {}
     child { node[treenode] (n24) {} }
      child {
        node[treenode] (n35) {}
          child { node[treenode] (n34) {} }
          child { node[treenode] (n45) {} }
      }
  }
  ;

  \node[above=3pt] at (n15) {$15$};
\node[left=3pt] at (n14) {$14$};
\node[right=3pt] at (n25) {$25$};
\node[left=3pt] at (n24) {24};
\node[right=3pt] at (n35) {$35$};
\node[left=3pt] at (n34) {$34$};
\node[below=3pt] at (n45) {$45$};
\end{tikzpicture}

\end{tabular}

\caption{Nested arc diagrams without forbidden patterns from Lemma \ref{lem: arcs_with_forbidden_patterns} and the corresponding full binary trees for $n=5$.}
\label{fig: arcs}
\end{figure}
\end{Example}
Returning to Gröbner bases, let us order the Plücker variables lexicographically:
\begin{equation}\label{eq: lexorder}
 p_{12} > p_{13} > \dots > p_{1n}>p_{23}>\dots>p_{2n}>\dots > p_{n-2,n}>p_{n-1,n}.   
\end{equation}
\begin{Proposition}\label{prop: GBgcomplete_grapg}
    The following $n\choose 4$ quadrics and  ${n\choose 5} + {n\choose 6}$ cubics form a reduced lexicographic Gröbner basis of the Plücker ideal $I_{L_n}$ for the variable order~\eqref{eq: lexorder}.
    \begin{equation}\label{eq: GB_complete_graph}
        \begin{aligned}
       Q_{ijkl}:= \underline{p_{ij}p_{kl}} - p_{ik}p_{jl} + p_{il}p_{jk}, \quad &\text{ for } i<j<k<l,\\
       C^{(5)}_{ijklm}:= \underline{p_{ik}p_{jl}p_{km}} - p_{ik}p_{jm}p_{kl} - p_{il}p_{jk}p_{km} + p_{im}p_{jk}p_{kl}, \quad &\text{ for } i<j<k<l<m,\\
       C^{(6)}_{ijklms}:= \underline{p_{il}p_{jm}p_{ks}} - p_{il}p_{js}p_{km} - p_{im}p_{jk}p_{ls} + p_{is}p_{jk}p_{lm}, \quad &\text{ for } i<j<k<l<m<s.
    \end{aligned}
    \end{equation}
\end{Proposition}
We will need the following proposition.
\begin{Proposition}\label{prop: StanleyReisner_of_Kn}
Let $M_{L_n}$ be the square-free monomial ideal generated by the ${n\choose 4}+ {n\choose 5} + {n\choose 6}$ underlined monomials in~\eqref{eq: GB_complete_graph}. Then we have
\begin{enumerate}
  \item $M_{L_n}$ is equidimensional;
  \item $\dim(M_{L_n}) = 2n-3$;
  \item $\deg(M_{L_n}) = C_{n-2}$.
\end{enumerate}
\end{Proposition}
\begin{proof}
This follows from the Stanley–Reisner correspondence and Lemma~\ref{lem: arcs_with_forbidden_patterns}.
\end{proof}
\begin{proof}[Proof of Proposition \ref{prop: GBgcomplete_grapg}]
The listed polynomials~\eqref{eq: GB_complete_graph} clearly generate $I_{L_n}$ since they contain all quadratic Plücker relations $Q_{ijkl}$, which minimally generate $I_{L_n}$. To prove that they form a Gröbner basis we apply Buchberger’s criterion. 
For the chosen lexicographic order~\eqref{eq: lexorder}, the underlined monomials in~\eqref{eq: GB_complete_graph} are indeed the leading terms. We now consider $S$-pairs of the quadrics $Q_{ijkl}$. Each quadric is determined by four indices $i<j<k<l$, so two distinct quadrics may share between zero and three indices.
\begin{itemize}
    \item If two quadrics share at most one index, or share two indices that lie in different Plücker variables of the leading terms, then their leading monomials are coprime. By Buchberger’s second criterion, the $S$-pair reduces immediately to zero.
    \item Suppose they share exactly two indices that occur together in one variable of the leading terms. Up to symmetry (note that $Q_{ijkl}$ is invariant under $(ik)(jl)$), we may assume that the last two indices coincide, so the quadrics are $Q_{ilms}$ and $Q_{jkms}$. Then
\[
S(Q_{ilms}, Q_{jkms}) 
= p_{jk}(-p_{im}p_{ls} + p_{is}p_{lm}) - p_{il}(-p_{jm}p_{ks} + p_{js}p_{km})
= C^{(6)}_{ijklms},
\]
and thus should be added to the generating set by Buchberger's algorithm.
\item Finally, if two quadrics share three indices, there are two possibilities up to a symmetry: the third common index is the smallest index of the corresponding variable $p$ or the biggest. These two cases both lead to $S$-pair of the form
 \begin{align*}
     C^{(5)}_{ijklm} & =  S(Q_{iklm}, Q_{jklm}) = p_{jk}(-p_{il}p_{km} + p_{im}p_{kl}) - p_{ik}(-p_{jl}p_{km} + p_{jm}p_{kl}) = \\
  & =    S(Q_{ijkm}, Q_{ijkl}) = p_{kl}(-p_{ik}p_{jm} + p_{im}p_{jk}) - p_{km}(-p_{ik}p_{jl} + p_{il}p_{jk}).
 \end{align*}
\end{itemize}
Before proving that this set is sufficient to form a Gröbner basis, we will reduce it.

\smallskip
\noindent
\textbf{Six-index cubics.}  
Consider $C^{(6)}_{ijklms} = S(Q_{ilms}, Q_{jkms})$ with $i<l<m<s$ and $j<k<m<s$.  
There are six possible relative orderings of $(i,j,l,k)$:
\begin{align*}
    i <j <k <l, & \quad j <i <k <l,\\
    i <j <l <k, & \quad j <i <l <k,\\
    i <l <j <k, & \quad j< k <i <l.
\end{align*}
In the first three orderings, the leading term of $C^{(6)}_{ijklms}$ is $p_{il}p_{jm}p_{ks}$; in the last three, it is $p_{im}p_{jk}p_{ls}$. In all cases except the two patterns $(ijkl)$ and $(jilk)$, the leading term is divisible by a quadratic leading term $p_{ab}p_{cd}$ with $a<b<c<d$, so it reduces with respect to some $Q_{abcd}$. Moreover, after complete reduction with respect to~\eqref{eq: GB_complete_graph}, one obtains zero. For instance, when $i<j<l<k$, we have
\[
C_{ijklms}^{(6)} - p_{jm}Q_{ilks} + p_{js}Q_{ilkm} - C_{ijlkms}^{(6)} - p_{im}Q_{jlks} + p_{is}Q_{jlkm} = 0,
\]
where $C^{(6)}_{ijlkms}$ belongs to the set~\eqref{eq: GB_complete_graph} since its indices are ordered increasingly.  Therefore the only non-redundant cases are $(ijkl)$ and $(jilk)$.
Since the cubics $C^{(6)}_{ijklms}$ are invariant under the permutation $(ij)(kl)$, we may restrict to those with $i<j<k<l$, which are precisely the ones listed in~\eqref{eq: GB_complete_graph}.  

\smallskip
\noindent
\textbf{Five-index cubics.}  
Next, consider
\[
C^{(5)}_{ijklm} = S(Q_{iklm}, Q_{jklm}) = S(Q_{ijkm}, Q_{ijkl}).
\]
If obtained from the first $S$-pair, then $i<k<l<m$ and $j<k<l<m$, leaving two possibilities: $i<j$ or $j<i$.  
If obtained from the second $S$-pair, the only freedom is swapping $l$ and $m$. However, these cubics satisfy the relations
\[
C^{(5)}_{ijklm} = - C^{(5)}_{jiklm} = - C^{(5)}_{ijkml},
\]
so they are invariant with respect to such simple reflections up to a sign. Therefore, it suffices to include $C^{(5)}_{ijklm}$ with $i<j<k<l<m$ in the Gröbner basis.  

We now pass to the monomial ideal denerated by the underlined monomials in \eqref{eq: GB_complete_graph}. This is exactly the ideal $M_{L_n}$ from Proposition \ref{prop: StanleyReisner_of_Kn}, so it is equidimensional, and its dimension and degree agree with those of the Plücker ideal $I_{L_n}$. Therefore, it follows from \cite[Lemma~3.1]{CHM23} that the polynomials in~\eqref{eq: GB_complete_graph} form a Gröbner basis for $I_{L_n}$ with respect to the lexicographic order~\eqref{eq: lexorder}.
\end{proof}

The order~\eqref{eq: lexorder} is an elimination order. 
Although the Gröbner basis itself is not quadratic, we can nevertheless show
that the corresponding elimination ideals are quadratically generated.
More precisely, we may delete \emph{any} initial segment of Plücker variables
with respect to the order~\eqref{eq: lexorder} and still obtain a
quadratically generated ideal.

We adopt the graph conventions from Section~\ref{sec: 3} and denote
by $I_G$ the elimination ideal
$
I_G = I_{L_n} \cap \mathbb{K}[p_{ij} : \{i,j\} \in E(G)].
$

\begin{Definition}
A graph $G$ is \emph{compatible with the elimination order~\eqref{eq: lexorder}} if for some $n \geq k > j \geq 1$ its complement $G^c$ is isomorphic to a graph with edges
\[
\{1,2\},\{1,3\},\dots,\{j,j+1\},\dots,\{j,k\}.
\]
\end{Definition}
\begin{Corollary}
If $G$ is compatible with the elimination order~\eqref{eq: lexorder}, then $I_G$ is generated by quadratic Plücker relations
\begin{equation}\label{eq: quadrics_from_edges}
    \underline{p_{ij}p_{kl}} - p_{ik}p_{jl} + p_{il}p_{jk}, 
    \quad i<j<k<l,\;\; \{ij,kl,ik,jl,il,jk\}\subseteq E(G).
\end{equation}
\end{Corollary}
\begin{proof}
By the Elimination Theorem \cite[Theorem 1.4.1]{GB2013}, the Gröbner basis for $I_G$ with respect to the lexicographic order~\eqref{eq: lexorder} consists of those quadrics and cubics from~\eqref{eq: GB_complete_graph} that involve only variables $p_{ij}$ with $\{i,j\}\in E(G)$.  

Any cubic $C^{(5)}_{ijklm}$ is a relation among edges of a complete graph $K_5$ without two edges indicated in the left diagram of Figure~\ref{fig: graphs for cubics}. Such a cubic can be generated by the quadrics $Q_{iklm}$ and $Q_{jklm}$, provided that $\{l,m\}\in E(G)$. For compatible graphs this is always the case: if $\{i,k\}\in E(G)$, then every edge smaller than $\{i,k\}$ in the order~\eqref{eq: lexorder} also lies in $E(G)$, in particular $\{l,m\}$ for $i<k<l<m$.  

The same reasoning applies to any cubic $C^{(6)}_{ijklms}$. If it lies in the ring $\KK[p_{ij}:\{i,j\}\in E(G)]$ for a compatible graph $G$, then it belongs to the ideal generated by the quadrics~\eqref{eq: quadrics_from_edges}, since it is generated by $Q_{ilms}$ and $Q_{jkms}$. Here the additional edge $\{m,s\}$ for the graph appearing in the right diagram of Figure~\ref{fig: graphs for cubics} is again guaranteed to belong to $E(G)$ as soon as $\{i,l\}\in E(G)$ for $i<l<m<s$.  
    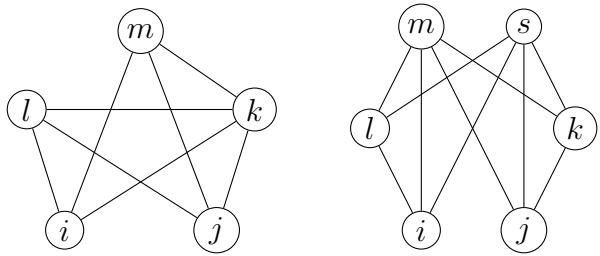
\begin{figure}
        \centering
        \begin{tikzpicture}[scale=2, every node/.style={circle, draw, inner sep=2pt}]
    \node (i) at (0,0) {$i$};
    \node (j) at (1,0) {$j$};
    \node (k) at (1.25,0.8) {$k$};
    \node (l) at (-0.25,0.8) {$l$};
    \node (m) at (0.5,1.32) {$m$};
    
    \draw (i) -- (k);
    \draw (i) -- (l);
    \draw (i) -- (m);
    \draw (j) -- (k);
    \draw (j) -- (l);
    \draw (j) -- (m);
    \draw (k) -- (l);
    \draw (k) -- (m);
\end{tikzpicture}
\qquad
\begin{tikzpicture}[scale=0.45, every node/.style={circle, draw, inner sep=2pt}]
    \node (i) at (1.5,0) {$i$};
    \node (j) at (4.5,0) {$j$};
    \node (l) at (0,3) {$l$};
    \node (k) at (6,3) {$k$};
    \node (m) at (1.5,6) {$m$};
    \node (s) at (4.5,6) {$s$};
    
    \draw (i) -- (s);
    \draw (i) -- (l);
    \draw (i) -- (m);
    \draw (j) -- (k);
    \draw (j) -- (s);
    \draw (j) -- (m);
    \draw (k) -- (s);
    \draw (k) -- (m);
    \draw (l) -- (s);
    \draw (l) -- (m);
\end{tikzpicture}
        \caption{Graphs supporting the cubic relations $C_{ijklm}^{(5)}$ and $C_{ijklms}^{(6)}$, respectively.}
        \label{fig: graphs for cubics}
    \end{figure}
\end{proof}

\section*{Acknowledgments}
\small{The research for this paper was initiated while the second author stayed at the Max Planck Institute for Mathematics in the Sciences, Leipzig, August 15 -- September 8, 2025. 

\medskip

The first author acknowledges support
from the European Research Council~\begin{footnotesize}(UNIVERSE PLUS, 101118787). 
$\!\!$ Views and opinions expressed
are however those of the authors only and do not necessarily reflect those of the European Union or the 
European
Research Council Executive Agency. Neither the European Union nor the granting authority
can be held responsible for them. \end{footnotesize}}


\begin{thebibliography}{99}
\bibitem{DEP80}
C.~De Concini, D.~Eisenbud and C.~Procesi, Young diagrams and determinantal varieties, {\em Inv. Math.} {\bf 56} (1980) 129--165.
\bibitem{CHM23}
A.~Conner, K.~Han and M.~Michałek, Sullivant--{T}alaska ideal of the cyclic {G}aussian {G}raphical {M}odel, {\tt arXiv.2308.05561} (2023).
\bibitem{Eisenbud80}
D.~Eisenbud, Introduction to algebras with straightening laws, {\em in} ``Ring Theory
and Algebra III'' (B.~R.~McDonald, Ed.), Proc. of the third Oklahoma Conf., Lect. Notes in Pure and Appl. Math. No. 55, Dekker, 1980, 243--268.
\bibitem{Grassmann}
H.~Grassmann, Die Ausdehnungslehre, (English
translation): ``Extension Theory'' (L. Kannenberg, Ed.), History of Mathematics, {\bf 19}, A.M.S., Providence, R.I. (2000).
\bibitem{EHH11}
V.~Ene, J.~Herzog and T.~Hibi, Cohen--Macaulay binomial edge ideals, {\em Nagoya Math. J.} {\bf 204} (2011), 57--68. 
\bibitem{HHgtm260} 
J.~Herzog and T.~Hibi, ``Monomial Ideals'', GTM 260, Springer, 2011.
\bibitem{HHZ04}
J.~Herzog, T.~Hibi and X.~Zheng, Monomial ideals whose powers have a linear resolution, {\em Math. Scand.} {\bf 95} (2004), 23--32.
\bibitem{H+10}
J.~Herzog, T.~Hibi, F.~Hreinsd\'ottir, T.~Kahle and J.~Rauh, Binomial edge ideals and conditional independence statements, {\em Advances in Appl. Math.} {\bf 45} (2010), 317--333.
\bibitem{Hibi87}
T.~Hibi, Distributive lattices, affine semigroup rings and algebras with straightening laws, 
\bibitem{GB2013}
T.~Hibi, Ed., ``Gr\"obner Basis'', Springer, 2013. 
\bibitem{OH99}
H.~Ohsugi and T.~Hibi, Toric ideals generated by quadratic binomials, {\em J. of Algebra} {\bf 218} (1999), 509--527.
\bibitem{Pluecker1}
J.~Plücker, System der Geometrie des Raumes in neuer analytischer Behandlungsweise, Düsseldorf, Schaub’sche Buchhandlung, (1846).
\bibitem{Pluecker2}
J.~Plücker, Über ein neues Coordinatsystem, {\em J. für die Reine und Ungew. Math.}, 
{\bf 5} (1830), 1-36 [Ges. Abh. n.9, 124–158].
\bibitem{Reisner76}
G.~Reisner, Cohen--Macaulay quotients of polynomial rings, {\em Advances in Math.} {\bf 21} (1976), 30--49.
\bibitem{Stanley15}
R.~P.~Stanley, ``Catalan Numbers'', {\em Cambridge University Press}, 2015.
\bibitem{Stanley75}
R.~P.~Stanley, The Upper Bound Conjecture and Cohen--Macaulay Rings, {\em Studies
in Appl. Math.} {\bf 54} (1975), 135--142.
\bibitem{Stanley78}
R.~P.~Stanley, Hilbert functions of graded algebras, {\em Advances in Math.} {\bf 28} (1978), 57--83.
\bibitem{Sturmfels96}
B.~Sturmfels,  ``Gr\"obner Bases and Convex Polytopes'', Amer. Math. Soc., Providence, RI, 1996.
\end{thebibliography}
\end{document}